\documentclass[a4paper]{amsart}

\usepackage[a4paper,top=3cm,bottom=2cm,left=3cm,right=3cm,marginparwidth=1.75cm]{geometry}

\usepackage{amsmath,amsfonts,amssymb,graphicx,amsthm,tikz,parskip,enumitem, fullpage, pgf, mathrsfs}
\usetikzlibrary{shapes,snakes,topaths}
\usetikzlibrary{decorations.markings}
\usetikzlibrary{arrows}
\usetikzlibrary{cd}

\newtheorem{thm}{Theorem}
\numberwithin{thm}{section}
\newtheorem{lem}[thm]{Lemma}
\newtheorem{prop}[thm]{Proposition}
\newtheorem{cor}[thm]{Corollary}
\newtheorem{defn}[thm]{Definition}
\newtheorem{example}[thm]{Example}
\theoremstyle{remark}
\newtheorem{remark}[thm]{Remark}

\newcommand{\R}{\mathbb{R}}
\newcommand{\RP}{\mathbb{R}P}
\newcommand{\CP}{\mathbb{C}P}
\renewcommand{\P}{\mathcal{P}}
\newcommand{\tildetrue}{\raise.17ex\hbox{$\scriptstyle\sim$}}
\newcommand{\Z}{\mathbb{Z}}
\newcommand{\into}{\hookrightarrow}
\newcommand{\ol}{\overline}

\newcommand{\TC}{\operatorname{TC}}
\newcommand{\cat}{\operatorname{cat}}
\newcommand{\genus}{\operatorname{genus}}

\title{Relative topological complexity of a pair}
\author{Robert Short}

\begin{document}

\begin{abstract}
For a pair of spaces $X$ and $Y$ such that $Y \subseteq X$, we define the relative topological complexity of the pair $(X,Y)$ as a new variant of relative topological complexity.  Intuitively, this corresponds to counting the smallest number of motion planning rules needed for a continuous motion planner from $X$ to $Y$.  We give basic estimates on the invariant, and we connect it to both Lusternik-Schnirelmann category and topological complexity.  In the process, we compute this invariant for several example spaces including wedges of spheres, topological groups, and spatial polygon spaces.  In addition, we connect the invariant to the existence of certain types of axial maps.
\end{abstract}

\maketitle

\section{Introduction}

Topological complexity ($\TC$) is an invariant introduced by Michael Farber in \cite{farberTC} connecting a motion planning problem in robotics with algebraic topology.  Intuitively, we think of topological complexity as the smallest number of ``rules'' needed to form a continuous motion planning algorithm on a topological space $X$.  For robotics, we think of $X$ as the configuration space of some robot, and our motion planning algorithm outputs paths between the configurations in $X$. Continuity then requires that the paths remain ``close'' when the configurations are ``close'' in some sense.  It turns out that $\TC(X)$ only depends on the topology of $X$, and so topologists study the invariant applied to various topological spaces in the abstract sense rather than as configuration spaces of specific robots.

In the years since its introduction, several different variations of topological complexity have been studied pertaining to different motion planning problems.  Of interest to us is relative topological complexity.  Here, we restrict which configurations are allowed to be starting and ending configurations, but we permit the path to move within a larger configuration space.  This variant is mentioned in Farber's book on the subject \cite{farber}, and it is used there to prove that $\TC(X)$ is a homotopy invariant.  In this paper, we restrict our attention further to a certain method for choosing starting and ending configurations.

Our invariant is motivated by the following motion planning problem.  Suppose there was a robot with configuration space $X$, and the robot is given to us in an arbitrary configuration within $X$.  Our goal is to plan the robot's motion to a configuration within some specified subset $Y \subseteq X$.  Then the relative topological complexity of the pair $(X,Y)$ is the smallest number of ``rules'' needed to form a continuous motion planning algorithm on $X$ where the paths must end in $Y$.  This restriction provides two major advantages.  First, we are able to develop some standard tools for estimating this value, which we do in Section 2.  Then, there are natural relationships between this value and both $\TC(X)$ and the Lusternik-Schnirelmann category of $X$ which we explore in Section 3.

In Section 4, we apply this new variant of relative topological complexity to pairs of real projective spaces.  In so doing, we draw a deep connection to the existence of certain axial maps.  We draw this connection explicitly in Theorem \ref{axial}.  This connection follows a similar logic to \cite{farberRP}, where Farber, Tabachnikov, and Yuzvinsky connect the immersion dimension of real projective spaces to their topological complexity using axial maps.

Finally, in Section 5, we apply this new variant to pairs of spatial polygon spaces.  These have been studied by Hausmann and Knutson in \cite{hausknut} as well as by Davis in \cite{davis}.  In our study, we introduce new notation for interesting submanifolds of the spatial polygon spaces for consideration using our variant.  We also compute the relative topological complexity for pairs of spatial polygon spaces, relying upon the symplectic structure of these spaces.

This work is a piece of the author's PhD thesis under the supervision of Don Davis.  We are grateful for his guidance and support throughout this process, and for giving productive comments on early drafts.  We would also like to thank Jean-Claude Hausmann, Jesus Gonzalez, Steve Scheirer, Alan Hylton, and Brian Klatt for various productive and interesting conversations over the course of this project.

\section{Basic Estimates}
We begin by reframing the intuitions established in the introduction in terms of the Schwarz genus of a fibration.  This was introduced by Schwarz in \cite{schwarz}.

\begin{defn}
Let $f:E\to B$ be a fibration.  The \textit{Schwarz genus of $f$}, denoted $\genus(f)$, is the smallest $k$ such that there exists $\{U_i \}_{i=1}^k$ an open cover of $B$ along with sections $s_i:U_i \to E$ of $f$.
\end{defn}

To apply this to topological complexity, note that there is a natural fibration $p:PX\to X\times X$ where $PX$ is the space of paths in $X$.  This fibration assigns to each path in $X$ the endpoints, i.e. $p(\sigma)=(\sigma(0),\sigma(1))$.  A section of this fibration is a way to assign a path to a pair of points in $X$, aligning this with a motion planning rule in the intuitive notion.  Formally, we are left with the following definition for $\TC(X)$.

\begin{defn}
Let $p:PX\to X\times X$ be the fibration defined above.  Then, the \textit{topological complexity of $X$} is the Schwarz genus of $p$, or in other words $\TC(X)=\genus(p)$.
\end{defn}

It is worth noting that the definition we are using is the unreduced version of topological complexity used by Farber in \cite{farber}.  Many researchers in this field use a reduced version of topological complexity where $\overline{\TC}(X) = \TC(X)-1$.  We will use the unreduced version throughout this paper.

In addition to $\TC(X)$, Farber also introduced a relative topological complexity for general subsets of $X\times X$.  Again, this is defined in terms of Schwarz genus, but here the motion planning rule is to only move between select pairs of points within $X \times X$.

\begin{defn}{\cite{farber}}
If $A \subseteq X \times X$, the \textit{relative topological complexity}, denoted $\TC_X(A)$, is the Schwarz genus of the pullback fibration over $A$ induced by the inclusion map.
\end{defn}

If we consider $A$ as the set of pairs of allowed configurations in the intuitive notion of relative topological complexity, then this tracks the smallest number of rules needed to move through $X$ where the pairs of starting and ending points must lie in $A$.

\subsection{Relative Topological Complexity of a Pair}

As a variant of relative topological complexity, we consider the following problem.  Suppose we had a specified set of target configurations $Y \subseteq X$.  We wish to determine the smallest number of rules needed to create a continuous motion planner from any configuration in $X$ to a configuration in $Y$.  This natural question is answered by our new variant.

\begin{defn}
Let $Y \subseteq X$.  Let $P_{X\times Y} = \{\gamma\in PX|\gamma(0)\in X \text{ and }\gamma(1)\in Y\}$.  There is a natural fibration $P_{X\times Y} \overset{\pi}{\longrightarrow} X \times Y$ where $\pi(\gamma)=(\gamma(0),\gamma(1))$.  Then, the \textit{relative topological complexity of the pair $(X,Y)$} is the Schwarz genus of $\pi$.  In other words, $\TC(X,Y)=\genus(\pi)$.
\end{defn}

One can also think of the fibration $\pi$ as the pullback of the usual topological complexity fibration induced by the inclusion map $X\times Y \into X\times X$.  Doing so, we can immediately get a convenient upper bound on $\TC(X,Y)$ thanks to a theorem from Schwarz.

\begin{prop}{\cite[Prop 7]{schwarz}}
Let $p:E\to B$ be a fibration and suppose $i:A\to B$ is a continuous map.  Let $i^*p:i^*E\to A$ be the pullback fibration over $A$ induced by $i$.  Then, $\genus(i^*p) \leq \genus(p)$.
\end{prop}

We can easily get the following corollary to this proposition.

\begin{cor}\label{RelTC<TC}
For $Y \subseteq X$, $\TC(X,Y) \leq \TC(X)$.
\end{cor}

This provides our first upper bound on relative topological complexity of a pair.  Both Farber and Schwarz give other methods for finding upper bounds of these values (in \cite{farber}, and \cite{schwarz} respectively).  We will refer to Schwarz here as we are framing our results primarily in terms of Schwarz genus.

\begin{thm}{\cite[Thm 5]{schwarz}}\label{schwarzub}
Let $F\to E \xrightarrow{p} B$ be a fibration where $\pi_j(F)=0$ for $j<s$.  Then, $$\genus(p)<\tfrac{\dim(B)+1}{s+1}+1$$
\end{thm}

Some easy corollaries of this are listed below:
\begin{cor}\label{schwarzubcor}
	Let $\pi:P_{X\times Y} \to X \times Y$ be the fibration defining $\TC(X,Y)$.  Suppose $\pi_j(X)=0$ for $j\leq s$.  Then:
	\begin{enumerate}
		\item $\TC(X,Y) <\tfrac{\dim(X)+\dim(Y)+1}{s+1}+1$
		\item $\TC(X,Y) \leq \dim(X)+\dim(Y)+1$
	\end{enumerate}
\end{cor}

\begin{proof}
For (1), note that in the fibration defining $\TC(X)$, the fiber is $\Omega X$, the loopspace of $X$.  Since $\pi:P_{X\times Y} \to X\times Y$ is the pullback of that fibration, it has the same fiber.  Thus, the condition that we need from Theorem \ref{schwarzub} is that $\pi_j(\Omega X) =0$ for $j<s$.  This is the same as $\pi_j(X)=0$ for $j\leq s$, which is as we assumed.

For (2), we need only notice that $\frac{\dim(X) + \dim(Y) +1}{s+1} +1 \leq \dim(X) + \dim(Y) +2$.  Thus, $\TC(X,Y) < \dim(X) + \dim(Y) +2$, so $\TC(X,Y) \leq \dim(X) + \dim(Y) +1$.
\end{proof}

In addition to upper bounds, we can use cohomology to determine lower bounds on these values.  Schwarz provides a cohomological lower bound on Schwarz genus in \cite{schwarz}, but Farber improves upon this when he provides a lower bound for $\TC(X)$ in \cite{farber}.  We follow Farber's example and prove a similar result here for $\TC(X,Y)$, and the proof adds some details that can be useful in computing lower bounds using cohomology.

\begin{defn}
The zero-divisors of the diagonal inclusion map are defined as $$Z(X\times Y) = \ker (H^*(X\times Y) \overset{\Delta_Y^*}{\longrightarrow} H^*(Y)).$$ 
\end{defn}

\begin{remark}
Notice that $P_{X\times Y} \simeq Y$ and this homotopy equivalence is induced by the map $c$ sending points in $Y$ to their constant maps in $P_{X\times Y}$.  Moreover, we can use the commutative diagram below to give an equivalent definition for $Z(X\times Y)$.

{\centering
\begin{tikzcd}
  & & P_{X\times Y} \arrow[d,"\pi"]\\
 Y \arrow[urr,"c",bend left=10] \arrow[r,"\Delta"] \arrow[rr,"\Delta_Y"',bend right] & Y\times Y \arrow[r, "\iota \times \text{id}"] & X \times Y 
\end{tikzcd}
\par}

Using the above diagram, we get that $Z(X\times Y) = \ker(H^*(X\times Y) \overset{\pi^*}{\longrightarrow} H^*(P_{X\times Y}))$.

Also, if we take cohomology with coefficients in a field, we can use the Kunneth theorem for cohomology to get that $Z(X \times Y) = \ker(H^*(X)\otimes H^*(Y) \overset{\iota^* \otimes id^*}{\longrightarrow} H^*(Y))$.
\end{remark}

\begin{thm}\label{cohomlb}
If there is a non-zero, $k$-fold product of elements in $Z(X \times Y)$, then we have that $\TC(X,Y) >k$.
\end{thm}

\begin{proof}
Assume $\TC(X,Y) \leq k$.  Then, take $\{U_j\}_{j=1}^k$ to be an open cover of $X\times Y$ such that for each $j$, there exists $s_j:U_j \to P_{X\times Y}$ with $s_j$ a section of $\pi$.  Then, for each $j$, we get the following commutative diagram:

{\centering
\begin{tikzcd}
\pi^{-1}(U_j) \ar[r,"a"] \ar[d,"\pi_j"] & P_{X\times Y} \ar[d,"\pi"] \\
U_j \ar[u,bend left,dashed, "s_j"] \ar[r,"b"] & X\times Y
\end{tikzcd}
\par}

This induces the following diagram in cohomology:

{\centering
\begin{tikzcd}
H^*(\pi^{-1}(U_j)) & \ar[l,"a^*"']  H^*(P_{X\times Y})  \\
H^*(U_j) \ar[u,"\pi_j^*"']& H^*(X\times Y) \ar[l,"b^*"'] \ar[u,"\pi^*"']
\end{tikzcd}
\par}

Since $\pi_j$ has a section, $\pi_j^*$ is injective.  So, if we take $\alpha \in Z(X\times Y)$, we have that $\pi^*(\alpha)=0 \implies a^*(\pi^*(\alpha))=0$.  By the diagram above, this implies that $\pi_j^*(b^*(\alpha))=0$, but $\pi_j^*$ in injective, so we see that $b^*(\alpha)=0$.  Thus, by exactness, $\alpha \in \operatorname{Im}(H^*(X\times Y, U_j) \to H^*(X\times Y))$.  We can then use this in the following diagram:

{\centering
\begin{tikzcd}
H^*(X\times Y\, , U_1) \otimes \cdots \otimes H^*(X\times Y\, , U_k) \ar[r] \ar[d] & H^*(X\times Y, \bigcup\limits_{j=1}^k U_j) = 0 \ar[d]\\
H^*(X\times Y) \otimes \cdots \otimes H^*(X\times Y) \ar[r,"\Delta^*"] & H^*(X\times Y)
\end{tikzcd}
\par}

Following the diagram, if $\alpha_1 \otimes \cdots \otimes \alpha_k \in H^*(X\times Y) \otimes \cdots \otimes H^*(X\times Y)$ is such that $\alpha_j \in Z(X\times Y)$, then $\alpha_1 \otimes \cdots \otimes \alpha_k$ pulls back to $\widetilde{\alpha}_1 \otimes \cdots \otimes \widetilde{\alpha}_k \in H^*(X\times Y\, , U_1) \otimes \cdots \otimes H^*(X\times Y\, , U_k)$.  By commutativity, we get that $\Delta^*(\alpha_1 \otimes \cdots \otimes \alpha_k) = 0$.

Thus, by contrapositive, if we have elements $\alpha_1, \cdots, \alpha_k \in Z(X\times Y)$ such that $\Delta^*(\alpha_1 \otimes \cdots \otimes \alpha_k) \neq 0$, we must have that $\TC(X,Y) >k$.

\end{proof}

The primary takeaway of the above result is that we can compute cohomological lower bounds using knowledge of the cup product in $Y$ and knowledge of the inclusion-induced map $\iota^*:H^*(X) \to H^*(Y)$.  Symplectic structures often exhibit useful knowledge of the inclusion-induced map in a powerful way.  Inspired by \cite[Thm 1]{farberRP}, we get the following theorem.

\begin{thm}\label{sympTC}
Let $(X,\omega_X)$ be a simply-connected, closed, symplectic manifold of dimension $2n$ with submanifold $Y$ of dimension $2m$ carrying a symplectic form $\omega_Y$ such that $\iota^*([\omega_X])=[\omega_Y]$.  Then $\TC(X,Y) = n+m+ 1$.
\end{thm}

\begin{proof}
First, note that $\TC(X,Y) \leq n+m+1$ via Corollary \ref{schwarzubcor}, since $X$ is simply-connected.

For the lower bound, consider the cohomology classes $[\omega_X]$ and $[\omega_Y]$.  Since $\iota^*[\omega_X] = [\omega_Y]$, $[\omega_X]\otimes 1 - 1 \otimes [\omega_Y]$ is a zero-divisor in $Z(X \times Y)$.  Expanding via the binomial theorem, $([\omega_X] \otimes 1 - 1 \otimes [\omega_Y])^{n+m} = (-1)^m\binom{n+m}{m}[\omega_X]^n \otimes [\omega_Y]^m \neq 0$.  Thus, $\TC(X,Y) > n+m$ by Theorem \ref{cohomlb}.  The result follows.

\end{proof}

As an example, notice that $\CP^n$ is a closed, simply-connected, symplectic manifold and also that $\CP^m \subset \CP^n$ is a symplectic submanifold where the natural inclusion map satisfies that $\iota^*([\omega_{\CP^n}]) = [\omega_{\CP^m}]$.  Thus, we get an easy corollary.

\begin{cor}\label{CPn}
$\TC(\CP^n,\CP^m) = n+m+1$.
\end{cor}

This corollary generalizes the result from \cite[Cor 2]{farberRP} that $\TC(\CP^n) = 2n+1$.  In general, projective spaces provide examples where the inclusion-induced map is well-behaved in cohomology.  We will return to the case of real projective spaces in Section 4.

\section{Relationship with Other Invariants}

The other two invariants we consider here are $\TC(X)$ and the Lusternik-Schnirelmann category (L-S cat) of $X$, denoted $\cat(X)$.  We defined $\TC(X)$ earlier, but we can use the Schwarz genus to give a definition for $\cat(X)$ that works well for our purposes.

For a pointed space $(X,x_0)$, there is a natural fibration $p_0:P_0X \to X\times \{x_0\}$ where $P_0X$ is the space $\{\sigma \in PX \, | \, \sigma(1)=x_0\}$.  Then $p_0(\sigma)=(\sigma(0),x_0)$ defines the fibration.  Notice that this is again the pullback of the fibration we used to define $\TC(X)$ over the inclusion map $X \times \{x_0\} \into X\times X$.

\begin{defn}
Let $p_0:P_0X\to X\times \{x_0\}$ be the fibration defined above.  Then, the \textit{Lusternik-Schnirelmann category of $X$} is the Schwarz genus of $p_0$, denoted by $\cat_{x_0}(X)$ or $\cat(X)$ if the basepoint is implied.
\end{defn}

\begin{remark}
The usual definition of $\cat(X)$ is the smallest number of sets $U_i \subseteq X$ needed to cover $X$ where $U_i$ is contractible in $X$.  In \cite[Thm 18]{schwarz}, Schwarz proves that when $p:E\to B$ is a fibration with $E$ contractible, then $\genus(p)=\cat(B)$.  In the above fibration, $P_0X$ is contractible, so this definition corresponds to the usual definition of L-S cat.  We will need a quick lemma to show that our definition is also independent of the choice of basepoint under reasonable conditions.
\end{remark}

\begin{lem}\label{catequiv}
If $X$ is path-connected, then for any $x_0, y_0 \in X$, $\cat_{x_0}(X)=\cat_{y_0}(X)$.
\end{lem}

\begin{proof}
It suffices to show that $\cat_{x_0}(X)\leq \cat_{y_0}(X)$ by symmetry.

Suppose $\cat_{y_0}(X)=k$.  Then there is an open cover of $X\times \{y_0\}$, say $\{U_i\}_{i=1}^k$ with sections $s_i:U_i \to P_0$.  To construct an open cover of $X\times \{x_0\}$, let $V_i=\{(u,x_0)\, | \, (u,y_0) \in U_i\}$.  Then, since $X$ is path-connected, there exists some path $\sigma$ such that $\sigma(0)=y_0$, and $\sigma(1)=x_0$.  The map $s_i':V_i \to P_0$ given by $s_i'(u,x_0)=s_i(u,y_0)\cdot \sigma$, where $\cdot$ denotes concatenation, is a continuous section of the fibration for $\cat_{x_0}(X)$.  Thus, $\cat_{x_0}(X)\leq k$.
\end{proof}

For a pointed space $(X,x_0)$, there is a natural inclusion $f:X\times \{x_0\} \into X\times Y$ when $x_0\in Y$.  Moreover, we have the following relationship:

\begin{prop}\label{cat<TC}
Assume $Y$ is a non-empty subset of a path-connected pointed space $(X,x_0)$.  Then we have $$\cat(X) \leq \TC(X,Y) \leq \TC(X) \leq \cat(X\times X).$$
\end{prop}

\begin{proof}

We saw in Corollary \ref{RelTC<TC} that $\TC(X,Y) \leq \TC(X)$, and in \cite[Prop 4.19]{farber}, Farber proves that $\TC(X) \leq \cat(X\times X)$, so all that is left for us is to show that $\cat(X) \leq \TC(X,Y)$.

By Lemma \ref{catequiv}, $\cat_{x_0}(X) = \cat_y(X)$ for some $y\in Y$ since $X$ is path-connected.  Suppose $\TC(X,Y)=k$, and that the sets $\{U_i\}_{i=1}^k$ exhibit this fact.  Then, define $V_i = U_i \cap (X \times \{y\})$.  Note that $\{V_i\}_{i=1}^k$ forms an open cover of $X \times \{y\}$ with sections $s_i|_{V_i}:V_i \to P_{X \times \{y\}}=P_0$.  Thus, $\cat(X)\leq k$.

\end{proof}

One comment on this result is that this yields an interpretation of the relative topological complexity of the pair $(X,Y)$ as a means of interpolating between TC and L-S cat.  In fact, if $Y \subset Z$, it is easy to see that $\TC(X,Y) \leq \TC(X,Z)$.

However, notice that it is not necessarily true that $\TC(Y) \leq \TC(X,Y)$.  This is because the pullback of the inclusion map $Y\times Y \into X\times Y$ has a total space of paths \emph{in X} between points in $Y$.  We can exhibit this fact computationally in the following example.

\begin{example}
Recall that $\RP^2$ embeds in $\R^4$, so we have an embedding of $\RP^2$ into $S^4$ by taking the one-point compactification of $\R^4$.  We know that $\TC(S^4,\RP^2) \leq \TC(S^4)=3$, but $\TC(\RP^2) = 4$.  Thus, $\TC(\RP^2) > \TC(S^4,\RP^2)$.
\end{example}

One nice application of this relationship occurs in the presence of a topological group.  Let $G$ be a path connected topological group.  As a simple exercise succeeding \cite[Prop 4.19]{farber}, Farber indicates that $\TC(G)=\cat(G)$.  Using this, we have the following corollary:

\begin{cor}\label{toplgrp}
Let $H$ be a non-empty subset of a path-connected topological group $G$.  Then $\TC(G,H) = \cat(G)$.
\end{cor}

\begin{proof}
$\cat(G) \leq \TC(G,H) \leq \TC(G) = \cat(G)$
\end{proof}

Since any torus $T^n=(S^1)^n$ has a topological group structure, and $\TC(T^n) = n+1$ is a well-known result, we can compute their relative topological complexity as a corollary to this:

\begin{cor}\label{torus}
Let $H\subseteq T^n$.  Then $\TC(T^n,H) = n+1$.  In particular, $\TC(T^n, T^m)=n+1$ when $n\geq m$.
\end{cor}

A standard result for both L-S cat and TC says that $\TC(X)=1$ if and only if $X$ is contractible (similarly $\cat(X)=1$ if and only if $X$ is contractible).  We now establish a similar result for relative topological complexity of the pair $(X,Y)$.

\begin{prop}\label{Xcontract}
$\TC(X,Y) = 1$ if and only if $X$ is contractible.
\end{prop}

\begin{proof}
We will prove both implications separately although the proofs are similar.

\begin{itemize}
\item [$\implies$:] Suppose $\TC(X,Y) = 1$.  Then $1 \leq \cat(X) \leq \TC(X,Y) =1$, so $\cat(X)=1$.  But $\cat(X)=1$ if and only if $X$ is contractible.

\item [$\impliedby$:] Suppose  $X$ is contractible.  Then $\TC(X)=1$.  Then $1\leq \TC(X,Y) \leq \TC(X)=1$, so $\TC(X,Y) =1$.
\end{itemize}
\end{proof}

It is also useful to think of what the contractibility of $Y$ can yield in terms of $\TC(X,Y)$ results.  This yields the following definition and theorem.

\begin{defn}
We say that a space $Y$ is \textit{contractible in $X$} if the inclusion map $\iota:Y\to X$ is homotopic to a constant map.
\end{defn}

\begin{thm}\label{Ycontract}
If $Y$ is contractible in $X$, then $\TC(X,Y) = \cat(X)$.
\end{thm}

\begin{proof}
We know $\cat(X) \leq \TC(X,Y)$, so all that remains is to see that $\cat(X) \geq \TC(X,Y)$ when $Y$ is contractible in $X$.

Suppose $\cat(X) = k$ and this is exhibited by an open cover $U_1, \dots, U_k$ of $X\times \{x_0\}$ with sections $s_i$ over each $U_i$.  Let $p_X(U_i)$ be the projection of $U_i$ onto its $X$ component.  Define $V_i = p_X(U_i)\times Y$.  Then, since $p_X(U_i)$ covers $X$, the collection of $V_i$ sets covers $X\times Y$.  Let $H:Y\times I  \to X$ be the homotopy where $H(Y,0)=x_0$ and $H(Y,1) = \iota(Y)$, and let $h(y)=H|_{\{y\}\times I}$ be the path from $x_0$ to $\iota(y)$ for $y\in Y$.  Define $s'_i(x,y) = s_i(x,x_0) \cdot h(y) $.  The pairs $(V_i,s'_i)$ form an open cover with sections for $\TC(X,Y)$ with $k$ elements.  Thus, $\cat(X) \geq \TC(X,Y)$.
\end{proof}

\subsection{Examples Involving Spheres}
Since $\pi_m(S^n)=0$ for $m<n$, the relative topological complexity of pairs of spheres is a simple corollary to Theorem \ref{Ycontract}.

\begin{cor}\label{sphere}
Take $n>m>0$, then $\TC(S^n,S^m) =\cat(S^n)=2$.
\end{cor}

\begin{remark}
It is beneficial to see an explicit motion planning algorithm exhibiting the fact that $\TC(S^n,S^m)=2$.  We provide this here as an example of the construction used in Theorem \ref{Ycontract}.

First we construct the open sets needed to see that $\cat(S^n)=2$.  Take the distinguished point of $S^n$ to be $e_1 = (1,0,\dots , 0)$, let its antipode be $e_2 = (-1,0, \dots, 0)$, and let $0<\varepsilon <1$.  Take $\pi_1:S^n \to \R$ to be projection onto the first component.  We define the open cover of $S^n \times \{e_1\}$ by $U_1 = \pi_1^{-1}((-\varepsilon,1]) \times \{e_1\} $ and $U_2 = \pi_1^{-1}([-1,\varepsilon)) \times \{e_1\}$.

For $i=1,2$, let $f_i:U_i \into S^n$ be the natural inclusion map.  We can easily define homotopies $F_i:U_i \times I \to S^n$ such that $F_i((x,e_1),0)=f_i(x,e_1)$ and $F_i((x,e_1),1) = e_i$.  Fix a path $\sigma:I \to S^n$ with $\sigma(0) = e_2$ and $\sigma(1)=e_1$.  Then, we can define sections over each $U_i$ by $s_1(x,e_1) = F_1((x,e_1),-)$ and $s_2(x,e_1)=F_2((x,e_1),-) \cdot \sigma$.

To incorporate $S^m$, we proceed exactly as we did in the proof of Theorem \ref{Ycontract}.  Let $\iota:S^m \to S^n$ be the inclusion map, and WLOG choose $h:S^m \times I \to S^n$ to be a homotopy where $h(S^m,0) = e_1$ and $h(S^m,1) = \iota(S^m)$.  Then for each $p \in S^m$, take $h(p) = h|_{\{p\}\times I}$ to be the path from $e_1$ to $\iota(p)$.  Take $V_1 = \pi_1^{-1}((-\varepsilon,1]) \times S^m$ and $V_2 = \pi_1^{-1}([-1,\varepsilon)) \times S^m$ mimicking $U_1$ and $U_2$ above so that $U_i \subseteq V_i$.  Then, define $s'_i(x,p) = s_i(x,e_1) \cdot h(p)$ for each $i$.  This exhibits the rules for $\TC(S^n,S^m)$ explicitly.
\end{remark}

Note that this differs significantly from the $\TC(S^n)$ case where the parity of the sphere's dimension determines the value.

We finish this section by putting all of the tools we developed to use on tackling pairs of bouquets of spheres.

\begin{prop}\label{wedge}
Suppose $(a_i)_{i=1}^n$ is a sequence of positive integers.  Then,
\begin{enumerate}
\item $\TC\bigg( \bigvee\limits_{i=1}^n S^{a_i}, * \bigg) = 2$; and
\item For $1 < m < n$, $\TC\bigg( \bigvee\limits_{i=1}^n S^{a_i}, \bigvee\limits_{j=1}^m S^{a_j} \bigg) = 3$.
\end{enumerate}
\end{prop}

\begin{proof}
For (1), without loss of generality, suppose $\iota(*) = x_0$  where $x_0$ is the wedge point of $\bigvee^n_{i=1} S^{a_i}$.  Thus, by Theorem \ref{Ycontract}, $\TC\big( \bigvee^n_{i=1} S^{a_i},*\big) = \TC\big( \bigvee^n_{i=1} S^{a_i}, x_0 \big) = \cat\big( \bigvee^n_{i=1} S^{a_i} \big) = 2$.

For (2), let $x_i$ denote the point in $S^{a_i}$ antipodal to $x_0$.  Let $C_a = \bigvee^n_{i=1} S^{a_i} - \{ x_i \}_{i=1}^n$ and let $C_b=\bigvee^m_{j=1} S^{a_j}-\{x_j\}_{j=1}^m$.  Notice that $C_a$ is contractible and $C_b \subseteq C_a$, so there exists a homotopy $h:C_a\times I \to C_a$ such that $h(x,-):I \to C_a$ is a path from $x$ to $x_0$ for each $x \in C_a$.  This homotopy also assigns paths from points in $C_b$ to $x_0$.  Also, for each $1 \leq i \leq n$, take $D_i$ to denote a contractible neighborhood of $x_i$ such that $x_0 \notin D_i$ and with contraction $k_i:D_i \times I \to D_i$ such that $k_i(x,0)=x$ and $k_i(x,1) = x_i$.  Finally, fix paths $\sigma_i:I \to \bigvee_{i=1}^n S^{a_i}$ where $\sigma_i(0) = x_i$ and $\sigma_i(1)=x_0$ for each $1 \leq i \leq n$.

We can now construct a motion planning algorithm on $\bigvee^n_{i=1}S^{a_i} \times \bigvee^m_{j=1}S^{a_j}$.  Let $\ol{\sigma}$ denote the path $\sigma$ traversed backwards.  Define $U_1 = C_a \times C_b$, $U_2 = \bigcup_{i=1}^n D_i \times C_b \, \cup \, \bigcup_{j=1}^m C_a \times D_j$, and $U_3=\bigcup_{(i,j) \in [n]\times [m]} D_i \times D_j$.  Let $X = \bigvee_{i=1}^n S^{a_i} \times \bigvee_{j=1}^m S^{a_j}$ and let $P_X \to X$ denote the relative topological complexity fibration.  Define $s_1:U_1 \to P_X$ by $s_1(x,y)=h(x,-)\cdot \ol{h}(y,-)$.  Each of $U_2$ and $U_3$ is a topologically disjoint union of open sets in $X$.  Then, we need only define sections over each set in the union and appropriately combine them for sections over $U_2$ and $U_3$.  We break these into the following three cases:

\begin{itemize}
\item For $D_i \times C_b$, use $s(x,y) = k_i(x,-)\cdot \sigma_i \cdot \ol{h}(y,-)$.
\item For $C_a \times D_j$, use $s'(x,y) = h(x,-)\cdot \ol{\sigma}_j \cdot \ol{k}_j(y,-)$.
\item For $D_i \times D_j$, use $s''(x,y) = k_i(x,-) \cdot \sigma_i \cdot \ol{\sigma}_j \cdot \ol{k}_j(y,-)$.
\end{itemize}

For the lower bound in the case where $\{b_j\} \neq \emptyset$, we must locate two non-zero cohomology elements in $H^*\big(\bigvee_{i=1}^n S^{a_i} \big)$.  Let $\pi_k:\bigvee_{i=1}^n S^{a_i} \to S^{a_k}$ denote the map sending the index $k$ sphere to $S^{a_k}$ and all other spheres to $x_0$.  For $0<k<n$, $\pi_k^*:H^*(S^{a_k}) \to H^*\big( \bigvee_{i=1}^n S^{a_i} \big)$ maps the generator ($gen$) of $H^*(S^{a_k})$ to a unique non-zero element in $H^*\big( \bigvee_{i=1}^n S^{a_i} \big)$.  Let $g_1 = \pi_1^*(gen)$ and $g_{m+1} =\pi_{m+1}^*(gen)$.

Let $\iota: \bigvee_{j=1}^m S^{a_j}\to \bigvee_{i=1}^n S^{a_i}$ be the inclusion map.  Then $(\iota \otimes id)^*(g_{m+1}\otimes 1)=0$ in $H^* (\bigvee_{j=1}^m S^{a_j} \times \bigvee_{j=1}^m S^{a_j})$.  Also, using the notation from before, we get that $\Delta^*(g_1 \otimes 1 - 1 \otimes g_1) = 0$.  Multiplying these two zero divisors together yields:
$$ (g_{m+1} \otimes 1)(g_1 \otimes 1 - 1 \otimes g_1) = g_{m+1}g_1\otimes 1 -g_{m+1}\otimes g_1 =-g_{m+1} \otimes g_1 \neq 0$$

Along with the motion planning algorithm above, this yields the result.

\end{proof}

\section{Real Projective Spaces}

Unlike complex projective spaces, real projective spaces do not have as simple or straightforward of a relationship with topological complexity.  Farber, Tabachnikov, and Yuzvinsky demonstrate this in \cite{farberRP}.  One of the main results from that paper connects the topological complexity of $\RP^n$ to the immersion dimension of $\RP^n$.  In particular, they show that:

\begin{thm}[\cite{farberRP}]\label{farberimm}
If $n \neq 1,3,7$, then $\TC(\RP^n)=Imm(\RP^n)+1$
\end{thm}

This connection is drawn in part using axial maps, a classical object of study in algebraic topology.  Determining the existence or nonexistence of axial maps using algebraic methods stretches back deep into the history of algebraic topology, see \cite{hopf}, \cite{ademgitlerjames}, and \cite{davisaxial}.  Moreover, James uses the connection between axial maps and immersions of real projective spaces in \cite{jamesimm} to get nonimmersion results.  In line with the connection $\TC(\RP^n)$ and its immersion dimension via axial maps in \cite{farberRP}, we can connect the relative topological complexity of pairs of real projective spaces to certain types of axial maps.  We give the relevant definition below.

\begin{defn}
Let $n$, $m$, and $k$ be integers such that $0 < m < n < k$.  A continuous map $g:\RP^n \times \RP^m \to \RP^k$ is called \textit{axial of type $(n,m,k)$} if the restrictions to $* \times \RP^m$ and $\RP^n \times *$ are homotopic to the respective inclusion maps in $\RP^k$.
\end{defn}

Note that this homotopy condition is equivalent to $g^*(x) = x \otimes 1 + 1 \otimes x \in H^*(\RP^m; \Z_2) \otimes H^*(\RP^n; \Z_2)$.  Following this definition, we will prove the following theorem.

\begin{thm}
\label{axial}
For $1 < m < n$, $\TC(\RP^n, \RP^m) = \min\{k \, | \, \text{there exists an axial map of type }(n,m,k-1)\}$.
\end{thm}

The proof of this theorem is complicated enough that it deserves its own subsection.  We include that here.

\subsection{Proof of Theorem \ref{axial}}

Let $\xi_n$ denote the canonical line bundle over $\RP^n$.  Recall that the external tensor product bundle $\xi_n \boxtimes \xi_m$ over $\RP^n \times \RP^m$ is defined by:

$$\frac{S^n \times S^m \times \R}{(x,y,t) \tildetrue (-x,y,-t) \tildetrue (x,-y,-t)} \xrightarrow{\xi_n \boxtimes \xi_m} \frac{S^n \times S^m}{(x,y)\tildetrue (-x,y) \tildetrue (x,-y)} \approx \RP^n \times \RP^m$$

We can also define the $k$-fold Whitney sum $k(\xi_n \boxtimes \xi_m)$ by replacing the $\R$ with $\R^k$ in the above definition.  By restricting to only the unit-length vectors in $k(\xi_n \boxtimes \xi_m)$, we can define the $(k-1)$-sphere bundle over $\RP^n \times \RP^m$ written as:

$$\frac{S^n \times S^m \times S^{k-1}}{(x,y,t)\tildetrue (-x,y,-t) \tildetrue (x,-y,-t)} \xrightarrow{S(k(\xi_n \boxtimes \xi_m))} \frac{S^n \times S^m}{(x,y) \tildetrue (-x,y) \tildetrue (x,-y)} \approx \RP^n \times \RP^m$$

We can relate relative topological complexity to the genus of $S(\xi_n \boxtimes \xi_m)$ in the following way:
\begin{lem}\label{TC>genus}
If $m<n$, $\TC(\RP^n,\RP^m)\geq \genus(S(\xi_n \boxtimes \xi_m))$
\end{lem}

\begin{proof}
Suppose $\TC(\RP^n,\RP^m)=k$, so there exists an open cover of $\RP^n \times \RP^m$, denoted $U_1, \dots ,U_k$ with sections $s_i:U_i \to P_{\RP^n\times \RP^m}$ of the relative TC fibration.  For $(x,y)\in U_i$, the path $s_i(x,y)$ on $\RP^n$ can be lifted to a path $\sigma_{xy}$ on $S^n$ such that following by the quotient map $q:S^n \to \RP^n$ yields $s_i(x,y)$.  Then, we can define $s_i':U_i \to \frac{S^n \times S^m \times \{-1,1\}}{(x,y,t)\tildetrue (-x,y,-t) \tildetrue (x,-y,-t)}$ by $s_i'(x,y) = [\sigma_{xy}(0),\sigma_{xy}(1),1]$.  This is well-defined since $s'_i(x,y) = [\sigma_{xy}(0),\sigma_{xy}(1),1]$.  So we have a continuous section of $S(\xi_n\boxtimes \xi_m)$.

\end{proof}

Once we notice that $S(k(\xi_n\boxtimes \xi_m))$ is the $k$-fold fiberwise join of $S(\xi_n \boxtimes \xi_m)$, we have the following easy corollary.

\begin{cor}\label{TC>genuscor}
If $m<n$, $\TC(\RP^n,\RP^m)\geq \min \{k\, | \, k(\xi_n\boxtimes\xi_m) \text{ has a nowhere-zero section}\}.$
\end{cor}

\begin{proof}
First, note that $k(\xi_n \boxtimes \xi_m)$ has a nowhere-zero section iff $S(k(\xi_n\boxtimes \xi_m))$ has a section.  By Theorem 3 in \cite{schwarz}, the genus of $S(\xi_n\boxtimes \xi_m))$ is the smallest $k$ such that the $k$-fold fiberwise join of $S(\xi_n \boxtimes \xi_m)$ has a section.  But the $k$-fold fiberwise join of $S(\xi_n \boxtimes \xi_m)$ is $S(k(\xi_n\boxtimes\xi_m))$.  Along with Lemma \ref{TC>genus}, this yields the result.

\end{proof}

Next we need to connect nowhere-zero sections of $k(\xi_n\boxtimes \xi_m)$ to non-singular maps as defined in \cite{farberRP}.  We reproduce the definition here.

\begin{defn}
A map $f:\R^n \times \R^m \to \R^k$ is \textit{non-singular} if for any $\lambda,\mu \in \R$, and any $(x,y) \in \R^n \times \R^m$, we have that:
\begin{itemize}
\item $f(\lambda x, \mu y) = \lambda \mu f(x,y)$, and
\item $f(x,y)=0 \implies x=0 \text{ or } y=0$.
\end{itemize}
\end{defn}

We connect the two ideas using the lemma below.

\begin{lem}\label{sect-nonsing}
If $n>m>1$ and there exists a nowhere-zero section of $k(\xi_n\boxtimes\xi_m)$, then there exists a non-singular map $\R^{n+1}\times \R^{m+1} \to \R^k$.
\end{lem}

\begin{proof}
Suppose $s$ is a nowhere-zero section of $k(\xi_n \boxtimes \xi_m)$.  Then, consider the following commutative diagram:

{\centering
	\begin{tikzcd}
		S^n \times S^m \times \R^k  \arrow[r,"q_1"] \arrow[d,"p_1"]& \frac{S^n\times S^m\times \R^k}{(x,y,t)\tildetrue (-x,y,-t)\tildetrue (x,-y,-t)} \arrow[d,"k(\xi_n\boxtimes \xi_m)"] \\
		S^n\times S^m \arrow[u,"s_2",bend left,dashed] \arrow[r,"q_2"] \arrow[ur,"s_1",dashed] & \frac{S^n \times S^m}{(x,y)\tildetrue(-x,y)\tildetrue(x,-y)}\arrow[u,"s",bend left,dashed]  
\end{tikzcd}
\par}

Let each $q_i$ be the natural quotient map, and let $p_1$ be the projection $(x,y,t)\mapsto (x,y)$.

Define $s_1=s\circ q_2$.  Notice that $s_1(x,y)=s_1(-x,y)=s_1(x,-y)$.

Now, $q_1$ defines a covering space, and since $S^n\times S^m$ is simply-connected, we can lift $s_1$ to some map $S^n\times S^m \to S^n\times S^m \times \R^k$.  This lift is not unique, but we can define $s_2$ as the unique lift of $s_1$ which is also a section of $p_1$.

Let $f$ be the $\R^k$ component of $s_2$, so that $s_2(x,y)=(x,y,f(x,y))$.  Notice that, in order for $s_2$ to be a lift of $s_1$, it must be that $q_1(s_2(x,y))=q_1(x,y,f(x,y))=s_1(x,y)=[x,y,f(x,y)]=[-x,y,-f(x,y)]=q_1(s_2(-x,y))$.  Thus, $f(-x,y)=-f(x,y)=f(x,-y)$ for any $(x,y)\in S^n\times S^m$.

We can then use $f$ to define a non-singular map $g:\R^{n+1}\times \R^{m+1}\to \R^k$ by $$g(x,y)=\left\{ \begin{array}{cl}
	|x||y|f(\tfrac{x}{|x|},\tfrac{y}{|y|}) &\text{if }x,y \neq 0\\
	0 &\text{if }x=0 \text{ or } y=0
	\end{array}
	\right.$$
\end{proof}

An easy corollary of Lemma \ref{sect-nonsing} and Corollary \ref{TC>genuscor} is the following.

\begin{cor}\label{sect-nonsingcor}
If $\TC(\RP^n,\RP^m)=k$ and $n>m>1$, then there exists a non-singular map $\R^{n+1}\times \R^{m+1}\to\R^k$.
\end{cor}

The last piece of this direction of the proof is to connect this result to axial maps.

\begin{lem}\label{nonsing-axial}
Assume $1 < m < n \leq k-1$.  There exists a bijection between non-singular maps $\R^{n+1}\times \R^{m+1} \to \R^{k}$ (identified under multiplication by a non-zero scalar) and axial maps of type $(n,m,k-1)$.
\end{lem}

\begin{proof}
Suppose $f: \R^{n+1}\times \R^{m+1} \to \R^k$ is non-singular.  Then, we can descend to a map $g: \RP^n \times \RP^m \to \RP^{k-1}$ by following the quotient maps.  This map is defined by sending an element $(u,v) \in S^n \times S^m$ to the line containing $f(u,v)$ in $\RP^{k-1}$.  To see that this is axial, consider $g|_{\RP^n \times *}$.  For a fixed $v \in S^m$, $g|_{\RP^n \times *}$ lifts to a function $\tilde{g}:S^n \to S^{k-1}$ such that $u \mapsto f(u,v)$.  Since $f(-u,v) = -f(u,v)$ by the non-singularity of $f$, we get that $g|_{\RP^n \times *}$ is not null-homotopic.  A similar argument shows that $g|_{*\times \RP^m}$ is also not null-homotopic.

For the other direction, suppose $g: \RP^n \times \RP^m \to \RP^{k-1}$ is an axial map of type $(n,m,k-1)$.  Passing to the universal covers, we have a continuous map $\tilde{g}:S^n \times S^m \to S^{k-1}$.  As above, $g$ being an axial map gives us that $\tilde{g}(-u,v) = -\tilde{g}(u,v) = \tilde{g}(u,-v)$ for any $(u,v) \in S^n \times S^m$.  Thus, we can extend $\tilde{g}$ to a non-singular map $f:\R^{n+1} \times \R^{m+1} \to \R^k$ defined by $$f(u,v) = |u| |v| \tilde{g}\bigg( \frac{u}{|u|} , \frac{v}{|v|}\bigg)$$

This yields the bijection.
\end{proof}

One benefit that this gives us is a method for choosing a non-singular map with some specific benefits.  We see this in the following corollary.

\begin{cor}\label{firstcoord}
Let $1 < m < n < k-1$ be integers such that there exists a non-singular map $\R^{n+1} \times \R^{m+1} \to \R^k$.  Then, there exists a non-singular map $f:\R^{n+1}\times \R^{m+1} \to \R^k$ such that for any $0 \neq u \in \R^{m+1}$, the first coordinate of $f((u,\overline{0}),u) \in \R^k$ is positive.
\end{cor}

\begin{proof}
For the given non-singular map that is assumed to exist, let $g: \RP^n \times \RP^m \to \RP^{k-1}$ be the corresponding axial map from Lemma \ref{nonsing-axial}.  By the axial map property, restricting to the diagonal $\RP^m \subseteq \RP^n \times \RP^m$ yields a null-homotopic function.  To see this quickly, note that $H^*(\RP^{k-1}) \overset{g^*}{\to} H^*(\RP^n) \otimes H^*(\RP^m) \overset{\iota^*\otimes 1}{\to} H^*(\RP^m) \otimes H^*(\RP^m) \overset{\Delta^*}{\to} H^*(\RP^m)$ sends the generator $x \in H^1(\RP^{k-1})$ to 0.  Thus, there is some $g' \simeq g$ such that $g':  \RP^n \times \RP^m \to \RP^{k-1}$ is constant along the diagonal.

By Lemma \ref{nonsing-axial}, $g'$ corresponds to some non-singular function $f: \R^{n+1} \times \R^{m+1} \to \R^k$.  By construction, $f(u,u)$ lies on a single line through the origin.  Via some rotation, we may assume that the first coordinate of $f(u,u)$ is positive, as desired.
\end{proof}

Finally, we require a way to point from non-singular maps to bounds on the relative topological complexity of the pair of real projective spaces.  For this, we again follow \cite{farberRP} with some modifications, using Corollary \ref{firstcoord} in a critical way.

\begin{lem}\label{nonsing>TC}
If there exists a non-singular map $\R^{n+1}\times \R^{m+1}\to \R^k$, then $\TC(\RP^n,\RP^m)\leq k$.
\end{lem}

\begin{proof}
Given a non-singular map $\rho:\R^{n+1}\times \R^{m+1}\to \R^k$, we can decompose $\rho$ into maps $\rho = (\rho_1, \dots , \rho_k)$ with $\rho_i:\R^{n+1}\times \R^{m+1}\to \R$ for each $i$.  We can also choose these so that $\rho_1(u,u)>0$ for any $u \in \R^{m+1}$ by Corollary \ref{firstcoord}.

Our goal is to create an open cover of $\RP^n \times \RP^m$ using the $k$ maps in the decomposition of $\rho$.  We will consider $\RP^n$ as the space of lines through the origin in $\R^{n+1}$ where $\RP^m$ is a natural subspace of $\RP^n$.

We construct the sets as follows.  For $2 \leq i \leq k$, define
$$U_i = \{(L,L')\in \RP^n \times \RP^m \, | \, L \neq L' \text{ and } \rho_i(u,u') \neq 0 \text{ for some } u\in L, u'\in L'\}.$$

For $i=1$, we have to do something a little different to guarantee we have pairs of lines $(L,L)$ covered.  Let

$$U_1 = \{(L,L') \in \RP^n \times \RP^m \, | \, \rho_1(u,u')\neq 0 \text{ for some }u\in L , u' \in L' \}.$$

Note that since $\rho$ is non-singular, each pair $(L,L') \in \RP^n \times \RP^m$ must fall into at least one of the $U_i$ sets.  Thus, $\{U_i\}_{i=1}^k$ forms an open cover of $\RP^n \times \RP^m$.

Next, we need sections of the relative topological complexity fibration over each $U_i$.  If $L\neq L'$, then, there exists a plane in $\R^{n+1}$ spanned by the two lines.  Once we orient this plane, we can move one line to the other by rotating in the plane along the direction of positive orientation.

If $(L,L') \in U_i$, then we can use $\rho_i$ to orient the plane.  Suppose $u \in L$ and $u\in L'$ are two unit vectors such that $\rho_i(u,u')>0$.  Then, define the positive orientation of the plane spanned by $L$ and $L'$ to be the direction given by moving $u$ to $u'$ through the angle smaller than $\pi$.  Then, we can define $s_i(L,L')$ to be the path moving $L$ in the positive orientation of this plane given by $\rho_i$.

When $L=L'$, which only occurs when $L \in \RP^m$, we can use the constant path.  Since this only occurs when $i=1$, we only need to make this distinction for $s_1$.  Using this, it is clear that each $s_i:U_i \to P_{\RP^n \times \RP^m}$ is a continuous section of the relative topological complexity fibration.

\end{proof}

Finally, we have the tools needed to prove our main result.

\begin{proof}[Proof of Theorem \ref{axial}]
By Corollary \ref{TC>genuscor} and Lemma \ref{nonsing>TC}, we have that when $1 < m < n$, $\TC(\RP^n, \RP^m) = \min \{k \, | \, \text{there exists a non-singular map } f: \R^{n+1}\times \R^{m+1} \to \R^k\}$.  Then, by Lemma \ref{nonsing-axial}, we can replace the non-singular map with an axial map of type $(n,m,k-1)$.

To complete this, we need only verify that $n+1 \leq \TC(\RP^n, \RP^m)$ to satisfy the conditions of Lemma \ref{nonsing-axial}.  But, it is well-known that $\cat(\RP^n) = n+1$, and $\TC(\RP^n,\RP^m) \geq \cat(\RP^n)$, so the condition holds.

\end{proof}

\section{Spatial Polygon Spaces}

Configuration spaces of polygons in $\R^3$ have been an interesting example in algebraic geometry for some time.  The configuration spaces of polygons in $\R^3$, called the spatial polygon spaces, come endowed with a symplectic structure which will prove useful to us later.  This structure has been studied by Klyachko in \cite{klyachko} as well as Kapovich and Millson in \cite{kapomill}.  We first encountered the spatial polygon spaces in the work of Jean-Claude Hausmann and Allen Knutson in \cite{hausknut}, but the topological complexity of spatial polygon spaces was not studied explicitly until the work of Don Davis in \cite{davis}.

These spatial polygon spaces are determined by the lengths of the sides of the polygons involved.  This motivates the following definition.

\begin{defn}
Let $\ell = (\ell_1, \dots , \ell_n) \in \R_+^n$ be a length vector of size $n$.  The \textit{spatial polygon space of $\ell$} is defined as:
$$\mathcal{N}(\ell) = \{ (z_1, \dots , z_n) \in (S^2)^n \, | \, \Sigma \ell_i z_i = \vec{0} \} / SO(3)$$
\end{defn}

We can think of $\mathcal{N}(\ell)$ as a set of ways to draw a polygon in $\R^3$ allowing for possible self-intersections.  A natural question to ask is which sides of the polygon we are capable of making collinear, or parallel.  We can refer to edges based on the index corresponding to its length in $\ell$, and doing this leads to a natural, and quite topologically valuable, definition for this case.  Take $[n] = \{1, \dots, n\}$.

\begin{defn}
A subset $S \subseteq [n]$ is \textit{short} (with respect to $\ell$) if $\sum\limits_{i \in S} \ell_i < \sum\limits_{j \notin S} \ell_j$.  Correspondingly, we say a subset $L \subseteq [n]$ is \textit{long} (with respect to $\ell$) if $\sum\limits_{i \in L} \ell_i > \sum\limits_{j \notin L} \ell_j$.
\end{defn}

Note that not every subset has to be short or long.  As an example, consider $\ell = (1,1,2,2)$.  Here, the subset $\{1,3\}$ is neither short nor long as $\ell_1 + \ell_3 = 3 = \ell_2 + \ell_4$.  However, when we have subsets like this, we can arrange the polygon into a configuration where all edges are collinear.  These collinear configurations create singularities in $\mathcal{N}(\ell)$, which can cause $\mathcal{N}(\ell)$ to fail to be a manifold.  To make sure we get a manifold, we will need to impose a few reasonable conditions on our length vectors.

\begin{defn}
Let $\ell$ be a length vector of size $n$.
\begin{enumerate}
\item We say $\ell$ is \textit{generic} if for any $S \subseteq [n]=\{ 1, \dots, n\}$ we have $\sum\limits_{i\in S} \ell_i \neq \sum\limits_{j \notin S} \ell_j .$

\item We say $\ell$ is \textit{non-degenerate} if for any $i\in [n]$ we have $\ell_i < \sum\limits_{j\neq i} \ell_j$.

\item We say $\ell$ is \textit{ordered} if $\ell_1 \leq \ell_2 \leq \dots \leq \ell_{n}$.
\end{enumerate}

\end{defn}

So long as $\ell$ is generic and non-degenerate, we can guarantee that $\mathcal{N}(\ell)$ is a manifold.  The topology of $\mathcal{N}(\ell)$ also respects permuting the order of the edges.  As stated precisely in \cite[1.4]{hausgeom}, for any $\sigma \in \Sigma_{n}$, let $\ell_{\sigma} = (\ell_{\sigma(1)}, \dots, \ell_{\sigma(n)})$; then $\mathcal{N}(\ell)$ is diffeomorphic to $\mathcal{N}(\ell_{\sigma})$.  That is, any length vector can be associated to an ordered length vector which generates the same topology.  As such, we can safely assume that our length vectors are ordered.  Finally, for a generic and non-degenerate length vector, every subset of $[n]$ is either short or long.

It will become necessary to have a way of sorting and categorizing the different short and long subsets of a length vector.  We use the following notation for this purpose.

\begin{defn}
$$\mathcal{S}_i(\ell) = \{ S \subseteq [n] \, | \, i \in S, \, S\text{ is short}\} \qquad \qquad \mathcal{S}(\ell) = \bigcup_{i=1}^n\mathcal{S}_i(\ell)$$
$$\mathcal{L}_i(\ell) = \{ L \subseteq [n] \, | \, i \in L, \, L\text{ is long} \} \qquad \qquad \mathcal{L}(\ell) = \bigcup_{i=1}^n\mathcal{L}_i(\ell)$$
\end{defn}

In \cite{hausknut}, Hausmann and Knutson give the following description for the cohomology ring of $\mathcal{N}(\ell)$ which uses short and long subsets in an essential way.

\begin{thm}\cite[Thm 6.4(2)]{hausknut}\label{hauscohom}
The cohomology ring of $\mathcal{N}(\ell)$ is given as
$$H^*(\mathcal{N}(\ell))=\Z[R,V_1, \dots , V_{n-1}]/\mathcal{I}$$
where $R,V_i \in H^2(\mathcal{N}(\ell))$, and $\mathcal{I}$ is the ideal generated by three families of relations:
\begin{enumerate}
\item $V_i^2 + RV_i$ for $i \in [n-1]$,
\item $\prod\limits_{i\in L}V_i$ for $L \in \mathcal{L}_n(\ell)$, and
\item $\sum\limits_{\overset{S \subset L}{S \text{ short}}}(\prod\limits_{i\in S}V_i)R^{|L-S|-1}$ for $L \in \mathcal{L}(\ell)-\mathcal{L}_n(\ell)$.
\end{enumerate}
\end{thm}

Hausmann and Knutson derive this description for the cohomology ring by studying and utilizing the symplectic structure of $\mathcal{N}(\ell)$.  They also describe a collection of natural $\operatorname{SO}(2)$-bundles over $\mathcal{N}(\ell)$ whose Chern classes prove particularly useful to us.

\begin{defn}\cite[\S 7]{hausknut}\label{bdledefn}
Let $\ell$ be a generic, non-degenerate length vector of size $n$.  Define $$A_j(\ell) = \{ (z_1, \dots , z_n) \in (S^2)^n \, | \, \Sigma \ell_i z_i = \vec{0} \text{ and } z_j = (0,0,1) \in S^2 \}.$$

Let $c_j(\ell) = c_1(A_j(\ell)) \in H^2(\mathcal{N}(\ell))$ denote the Chern class of the bundle $A_j(\ell) \to \mathcal{N}(\ell)$.
\end{defn}

Hausmann and Knutson then provide a method for describing each $c_j(\ell)$ using their description for $H^*(\mathcal{N}(\ell))$.

\begin{prop}\cite[Prop 7.3]{hausknut} \label{cherndesc}
In $H^2(\mathcal{N}(\ell))$, one has
\begin{itemize}
\item $c_j(\ell) = R + 2V_j$ if $i < n$; and
\item $c_n(\ell) = -R$.
\end{itemize}
\end{prop}

Hausmann and Knutson use these Chern classes to determine a very nice expression for the cohomology class for the symplectic form of $\mathcal{N}(\ell)$.

\begin{prop}[\cite{hausknut}, Remark 7.5]\label{haussymp}
If $\ell \in \Z^n$, then the symplectic form $[\omega] \in H^2(\mathcal{N}(\ell))$ is given by $$[\omega] = \sum\limits_{i=1}^n \ell_i c_i(\ell)$$
\end{prop}

Finally, it is well-known that $\mathcal{N}(\ell)$ is a simply-connected manifold of dimension $2(n-3)$ when $\ell$ is of size $n$ (see \cite[\S 1]{hausknut}, \cite[Lemma 10.3.33]{hausbook}).  Since it is a closed, symplectic, simply-connected manifold, we can compute $\TC(\mathcal{N}(\ell))$ as a corollary of \ref{sympTC} with $m=n$.  This is computed directly using Theorem \ref{hauscohom} in \cite{davis}.

\begin{prop}
For $\ell$ generic and non-degenerate of size $n$, $\TC(\mathcal{N}(\ell)) = 2n-5$.
\end{prop}

\subsection{Edge-Identifying Submanifolds}
There is a natural way to form submanifolds within $\mathcal{N}(\ell)$ by restricting our attention to configurations where selected edges are aligned together.  This space of configurations where selected edges are aligned forms a submanifold of $\mathcal{N}(\ell)$ which is homeomorphic, in some cases, to $\mathcal{N}(\ell^\P)$ for a different, but related length vector $\ell^\P$.

\begin{defn}\label{eilvdefn}
Let $\ell = (\ell_1, \dots , \ell_n)$ be a length vector and let $\P = (\P_1, \dots, \P_m)$ be an ordered set partition of $[n]$ into $m$ parts.  We say that a \textit{edge-identified length vector} from $\ell$ is a vector $\ell^\P = (\ell^\P_1, \dots , \ell^\P_m)$ such that $$\ell^\P_k = \sum\limits_{i\in \P_k}\ell_i.$$
\end{defn}

As this is a novel method of describing these spaces, we present some examples of edge-identified length vectors.

\begin{example}\label{eilvexs}
Let $\ell = (1,1,2,3,5,7)$.  Note that $\ell$ is a generic, non-degenerate, ordered length vector.  We give four examples of edge-identified length vectors of $\ell$.

\begin{itemize}
\item $\ell^{\P'} = (1,2,3,6,7)$:  Here, we have combined $\ell_2$ and $\ell_5$ into a single edge.  Explicitly, the ordered set partition is $\P' = (\{1\}, \{3\}, \{4\}, \{2,5\}, \{6\})$ giving us that $\ell^{\P'}_4 = \ell_2 + \ell_5$.

\item $\ell^{\P''}=(1,1,3,5,9)$:  Here, we have combined $\ell_3$ and $\ell_6$ into a single edge.  Explicitly, the ordered set partition is $\P'' = (\{1\}, \{2\}, \{4\}, \{5\}, \{3,6\})$ giving us that $\ell^{\P''}_5 = \ell_3 + \ell_6$.  Notice that we can identify other edges with the last edge $\ell_6$ as we do in this example.

\item $\ell^{\P'''}=(4,7,8)$:  Here, we have combined several of the edges together.  Explicitly, the ordered set partition is $\P'''=(\{1,4\}, \{6\}, \{2,3,5\})$ giving us that $\ell^{\P'''}_1 = \ell_1 + \ell_4$, and $\ell^{\P'''}_3 = \ell_2 + \ell_3 + \ell_5$.  Notice that it is possible, as in this example, to supplant the largest length by identifying other edges.  We can always permute the ordered set partition in order to end up with an ordered edge-identified length vector if we desire this.

\item $\ell^{\P''''}=(1,1,7,10)$:  Here, we have combined $\ell_3$, $\ell_4$, and $\ell_5$ into a single large edge.  In fact, $\ell^{\P''''}_4 = \ell_3 + \ell_4 + \ell_5 > \ell^{\P''''}_1 + \ell^{\P''''}_2 + \ell^{\P''''}_3$, giving us a degenerate length vector.  Thus, edge-identified length vectors need not preserve non-degeneracy in general.
\end{itemize}
\end{example}

Notice that all edge-identified length vectors preserve genericity, but they could fail to preserve non-degeneracy.  If we assume $\ell$ is non-degenerate, then we can preserve non-degeneracy by only identifying edges whose indices form short subsets.

The core concern in the above examples is which length in $\ell$ is assigned to a particular length in $\ell^\P$.  We can encode this in the function $\phi:[n] \to [m]$ given by $\phi(i)=j \iff i \in \P_j$ where $\P_j$ is the $j$th set of $\P$.  This function controls the topology of $\mathcal{N}(\ell^\P)$, but it also controls the inclusion map $\mathcal{N}(\ell^\P) \into \mathcal{N}(\ell)$.  We see this in the following proposition.

\begin{prop}\label{inducedcohom}
Let $\ell$ be a generic, non-degenerate length vector with $\ell^\P$ a non-degenerate edge-identified length vector of $\ell$.  Then the inclusion induced map $\iota^*:H^*(\mathcal{N}(\ell)) \to H^*(\mathcal{N}(\ell^\P))$ acts on the Chern classes by $\iota^*(c_j(\ell)) = c_{\phi(j)}(\ell^\P)$.
\end{prop}

\begin{proof}
Considering $\iota: \mathcal{N}(\ell^\P) \to \mathcal{N}(\ell)$, we know that $\iota^*(c_j(\ell)) = \iota^*(c_1(A_j(\ell))) = c_1(\iota^*(A_j(\ell)))$.  Thus, what we need to show is that the pullback bundle $\iota^*(A_j(\ell)) = A_{\phi(j)}(\ell^\P)$.

We can think of $A_j(\ell)$ as the space of polygons in $\R^3$ with the $j$th edge parallel to the $z$-axis.  Since $\mathcal{N}(\ell^e)$ identifies $\ell_j$ as part of $\ell_{\phi(j)}^\P$, $\iota^*(A_j(\ell))$ has all the edges in $\ell_{\phi(j)}^\P$ parallel to the $z$-axis.  Thus, $\iota^*(A_j(\ell)) = A_{\phi(j)}(\ell^\P)$.  And so, $\iota^*(c_j(\ell)) = c_1(A_{\phi(j)}(\ell^\P)) = c_{\phi(j)}(\ell^\P)$.

\end{proof}

With this information, we can determine the relative topological complexity for pairs of spatial polygon spaces.

\begin{thm}\label{NlTC}
Let $\ell$ be a generic, non-degenerate length vector with $\ell^\P$ a non-degenerate edge-identified length vector of $\ell$ as in Definition \ref{eilvdefn}.  Then, $\TC(\mathcal{N}(\ell),\mathcal{N}(\ell^\P))=n+m-5$.
\end{thm}

\begin{proof}
First, notice that $(\mathcal{N}(\ell),\omega)$ is a simply-connected symplectic manifold of dimension $2(n-3)$ and $(\mathcal{N}(\ell^\P), \omega_\P)$ is a submanifold of dimension $2(m-3)$ with its own symplectic structure.  We need only verify that $\iota^*([\omega]) = [\omega_\P]$, and then Theorem \ref{sympTC} yields the result.  We show this in the following computation.

\begin{eqnarray*}
\iota^*([\omega]) =& \iota^*\bigg( \sum\limits_{i=1}^n \ell_i c_i(\ell) \bigg) & \text{  (by Proposition \ref{haussymp})}\\
=&\sum\limits_{i=1}^n \ell_i \iota^*c_i(\ell) & \\
=&\sum\limits_{i=1}^n \ell_i c_{\phi(i)}(\ell^\P) &\text{  (by Proposition \ref{inducedcohom})} \\
=&\sum\limits_{j=1}^m \ell_j^\P c_{j}(\ell^\P) & \\
=&[\omega_\P] & \text{  (by Proposition \ref{haussymp})}
\end{eqnarray*}
\end{proof}

\end{document}